\documentclass[a4paper,11pt]{amsart}

\usepackage{graphicx}
\usepackage{pst-grad}
\usepackage{url}
\usepackage{pstricks}
\usepackage[utf8]{inputenc}
\usepackage{pstricks-add}
\usepackage{pgf}
\usepackage{tikz-cd}
\usepackage[thinlines]{easytable}
\usepackage[all]{xy}
\usepackage{pinlabel}
\usepackage{psfrag}
\usepackage[all]{xy}
\usepackage{hyperref}
\usepackage{float}

\usepackage{amsmath, amsfonts, amsthm, amssymb, amscd, mathtools, tensor}
 \newtheorem{thm}{Theorem}[section]
 \newtheorem{prop}[thm]{Proposition}
 \newtheorem{lem}[thm]{Lemma}
 \newtheorem{cor}[thm]{Corollary}
\theoremstyle{definition}
 \newtheorem{rem}[thm]{Remark}
 \newtheorem{definition}{Definition}[section]
\numberwithin{equation}{section}

\newcommand\mj{\mbox{\bf 1}}

\def\d#1{{#1\kern-0.4em\char"16\kern-0.1em}}
\def\D#1{{\raise0.2ex\hbox{-}\kern-0.4em#1}}
\def \Dj{\mbox{\raise0.3ex\hbox{-}\kern-0.4em D}}
\definecolor{britishracinggreen}{rgb}{0.0, 0.26, 0.15}
\def\zn{,\kern-0.09em,}
\definecolor{ruza}{rgb}{1.0, 0.01, 0.24}
\title{A calculus for $S^3$-diagrams of manifolds with boundary}

\author[Femi\' c]{Bojana Femi\' c}
\address{\scriptsize{Mathematical Institute SANU\\ Knez Mihailova 36, p.f.\ 367\\ 11001 Belgrade, Serbia}}
\email{ femicenelsur@gmail.com}
\author[Gruji\' c]{Vladimir Gruji\' c}
\address{\scriptsize{Faculty of Mathematics, University of Belgrade}}
\email{vgrujic@matf.bg.ac.rs}
\author[Obradovi\' c]{Jovana Obradovi\' c}
\address{\scriptsize{Mathematical Institute SANU\\ Knez Mihailova 36, p.f.\ 367\\ 11001 Belgrade, Serbia}}
\email{obradovic@karlin.mff.cuni.cz}
\author[Petri\' c]{Zoran Petri\' c}
\address{\scriptsize{Mathematical Institute SANU\\ Knez Mihailova 36, p.f.\ 367\\ 11001 Belgrade, Serbia}}
\email{zpetric@mi.sanu.ac.rs}

\date{}

\begin{document}

\begin{abstract}
We introduce a diagrammatic language for compact, orientable 3-dimensional manifolds with boundary. A diagrammatic calculus  appropriate for this language is introduced and its completeness is proved in the paper. Moreover, a calculus consisting of a finite list of local moves is presented.

\vspace{.3cm}


\vspace{.5ex}

\noindent {\small {\it Keywords$\,$}: knots, links, surgery, Kirby's calculus, 3-manifolds with boundary, gluing}
\end{abstract}

\maketitle
\section{Introduction}

The aim of this paper is to introduce a calculus for a presentation of compact, orientable, connected 3-manifolds with boundary in terms of diagrams embedded in $S^3$ in a form akin to the standard surgery presentation of closed, orientable, connected 3-manifolds. Some authors place such diagrams in different compact, connected 3-manifolds other than $S^3$ (see for example \cite{R97}). This is the reason for the name ``$S^3$-diagrams'' in the title of the paper. However, we abbreviate it to ``diagrams'' throughout the text.  Our motivation to introduce such a presentation of manifolds is to give a completely combinatorial description of the category 3Cob\footnote{By this we mean a diagrammatic presentation of cobordisms and a calculus for composition of diagrams.}, whose arrows are 3-dimensional cobordisms, which is an ongoing project. We hope this could support further investigations of faithfulness of 3-dimensional Topological Quantum Field Theories. On the other hand, the calculus could be applied within some coherence results in Category theory (see \cite[Section~9]{DPZ}).

That every closed, orientable, connected 3-manifold may be obtained by surgery on a link in $S^3$   was proved by Wallace, \cite{W60} and (independently) by Lickorish, \cite{L62}. This result provides a language for presentation of such manifolds. The rational surgery calculus for this language was introduced by Rolfsen, \cite[Chapter 9.H]{R76}. This calculus consists of two types of modifications and he proved that two surgery descriptions yield homeomorphic 3-manifolds if one can be transformed into the other by a finite sequence of these modifications. At about the same time, Kirby, \cite{K78} introduced another surgery calculus and he proved its completeness, i.e.\ that two surgery descriptions (with integral framing) yield homeomorphic 3-manifolds if and only if one can be transformed into the other by a finite sequence of operations from this calculus.
By relying on Kirby's result, Rolfsen, \cite{R84}, proved the completeness of his calculus.

Fenn and Rourke, \cite{FR79}, merged {{Kirby's   two operations}} into an infinite list of integral moves of one type, which is a special case of Rolfsen's second modification. Roberts, \cite{R97}, developed a calculus for surgery data in arbitrary compact, connected 3-manifold (possibly with boundary, or non-orientable). This calculus consists of moves of three types and he proved its completeness.

The first part of our work uses a generalization of Wallace-Lickorish result to compact, orientable, connected 3-manifolds with boundary in order to establish a diagrammatic language of these manifolds.
This language is based on the well-known language for closed manifolds that consists of surgery data in $S^3$ written in terms of framed links. We extend the ``alphabet'' by introducing some rigid ``symbols'' in the form of wedges of circles. The intuition behind a wedge of circles in a diagram is that its neighbourhood is removed from $S^3$ forming one component of the boundary. By the \emph{neighbourhood} of a wedge of circles we mean its regular neighbourhood in terminology of \cite[Definition~1.4]{J07}, which is appropriate for piecewise-linear category, or its graphical neighbourhood in terminology of \cite[Definition~6]{FH19}, which is appropriate for smooth category.

The second part adapts Roberts' calculus into a diagrammatic calculus adequate for our language. This adaptation is akin to the adaptation of Kirby's calculus made by Fenn and Rourke. As the Fenn and Rourke local moves can be reduced to a finite list, which is shown by Martelli, \cite{M12}, our calculus is also presentable by a finite list of local moves. Finally, we develop a rational surgery calculus for our language. We prove the completeness for all the calculi.

We do not specify the category in which we work. The results hold either in topological, PL or smooth category.
From now on, by a \emph{manifold} we mean a compact, connected and oriented 3-manifold possibly with boundary.  {{Unless explicitly stated otherwise, all homeomorphisms between such manifolds are assumed to be orientation-preserving.}} We take the orientation of $S^3$ fixed throughout the paper. A disk is always a 2-dimensional disk and a ball is a 3-dimensional ball. For a solid torus $S^1\times D^2$ we call $\{\ast\}\times D^2$ its \emph{meridional disk} and the boundary of this disk, i.e.\ $\{\ast\}\times S^1$, a \emph{meridian} of this torus.

Let $\Sigma$ be a closed surface such that $\Sigma\subseteq \partial M\cap \partial N$, for manifolds $M$ and $N$. (The orientation of $\Sigma$ is induced by the orientation of one of the manifolds and it is opposite to the orientation induced by the other.) Let $\varphi$ and $\theta$ be orientation-preserving automorphisms of $\Sigma$. We denote by $M\tensor*[^{}_\varphi]{+}{^{}_\theta}N$ the result of gluing $M$ and $N$ along $\varphi$ and $\theta$, i.e., the manifold $(M\sqcup N)/_\sim$, where $M\sqcup N=(M\times\{1\})\cup(N\times\{2\})$ and $\sim$ is such that
\begin{equation}\label{sim}
   \forall x\in\Sigma \quad (\varphi(x),1)\sim (\theta(x),2).
\end{equation}
When $\varphi$ or $\theta$ are omitted, this means that they are identities.\pagebreak

\section{The language}\label{diagrammatics}

{With the goal of introducing our language, we start this section by synthesizing key results on  surgery presentations of manifolds. For basic notions related to these  results, we refer the reader  to Chapter 9 of Rolfsen's classic book \cite{R76}.  We shall also rely on more specific results from the literature, which will be cited where appropriate.}

Consider a finite set of wedges of circles in $S^3$. By removing neighbourhoods of these wedges, one obtains a manifold $S^3_-$. Consider a framed link in $S^3_-$ and perform the surgery along this link, i.e.\ remove the tubular neighbourhood of every link component and sew it back according to the corresponding framing. {Since the surgery is performed in the interior of $S^3_-$,   the result is a manifold whose boundary is canonically identified (via the identity map) with the original $\partial(S^3_-)$.} Hence, a diagram in $S^3$ consisting of a collection of wedges of circles, together with a framed link (see Figure~\ref{dijagram1}) presents a manifold. (For the sake of better visualisation we mark the wedges of circles in red, although they could be distinguished from the components of the link, even in the case when a wedge consists of a single circle, since such a circle is unlabeled.)
\begin{figure}[h!]
    \centerline{\includegraphics[width=.5\textwidth]{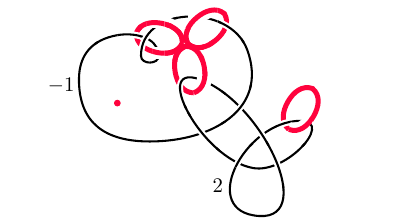}} \caption{A diagram}\label{dijagram1}
\end{figure}
In fact, every manifold is presentable in such form, and we explain this in the sequel.

Note that our presentation includes something more than just the homeomorphism type of a manifold. As we mentioned above, it gives the canonical identification of the boundary of the manifold with the boundary of $S^3_-$. For example, consider the following two diagrams.
\begin{figure}[H]
    \centerline{\includegraphics[width=0.8\textwidth]{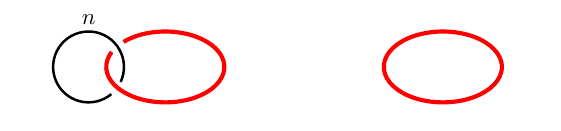} }
    \caption{Two diagrams}\label{d100}
\end{figure}
The left-hand side and the right-hand side diagram both present the solid torus, but there is no homeomorphism between the resulting manifolds that fixes the points of the common boundary (obtained by canonical identifications). This is the reason to introduce the notion of $\partial$-equivalent manifolds below. The calculus introduced in Section~\ref{calculus} takes care to distinguish such diagrams.

Let $C'$ be a disjoint union of balls and manifolds of the form $\Lambda\times I$ for $\Lambda$ a closed, orientable, connected surface. Let $\phi_i\colon D_i\to D_i'$, $i\in\{1,\ldots,m\}$ be homeomorphisms of disjoint disks, where each disk is either in the boundary of a ball component or in
a $\Lambda\times \{1\}$ boundary component of $C'$. A \emph{compression body} $C$ is the connected orientable manifold obtained from $C'$ by gluing its components along maps $\phi_i$. (Note that the disks and the gluing maps cannot be arbitrary according to demands for connectedness and orientability.) The boundary of $C$ consisting of the components $\Lambda\times \{0\}$ is denoted by $\partial_- C$ and the rest of the boundary of $C$, which is always closed, orientable and connected surface, is denoted by $\partial_+ C$. This definition of compression body is not standard but it is justified via \cite[Lemma~2.33]{J07} or \cite[Remark~3.1.3~(3)]{SSS}. 

For a compression body $C$ we consider its unlinked and unknotted embedding in $S^3$. Denote with $\Sigma$ the surface $\partial_+ C$ and let $H$ be the handlebody embedded in $S^3$ whose boundary is $\Sigma$ and $C\cap H=\Sigma$. In this way we have that $C\cup H$ is $S^3$ with some unlinked and unknotted handlebodies cut out. Let the orientation of $\Sigma$ be induced by the orientation of $H$, which agrees with the orientation of $S^3$.

By \cite[Theorem~3.1.10]{SSS} for every manifold $M$ there exists a compression body $C$ and an orientation-preserving homeomorphism $\theta\colon\Sigma\to\Sigma$ such that $\partial M=\partial_-C$, and $M=C+_\theta H$. Honestly, these equalities are orientation-preserving homeomorphisms, but for the sake of simplicity, we assume that $M$ is $C+_\theta H$, and that $\partial M$ is canonically identified with $\partial_-C$.

Let $\Xi$ be a closed surface, which is a part of the boundary of a manifold $M$. For a manifold $M'$ having $\Xi$ as a part of its boundary, we say that $M$ and $M'$ are $\Xi$-\emph{equivalent} when there exists an orientation-preserving homeomorphism $w\colon M\to M'$ such that the following diagram commutes:
\begin{center}
\begin{tikzcd}
\Xi\arrow[hook]{r}\arrow[hook]{dr}
&M \arrow{d}{w}\\
&M'.
\end{tikzcd}
\end{center}
This means that $w$ keeps the points of $\Xi$ fixed. If $\Xi$ is the common boundary of two $\Xi$-equivalent manifolds, then we say they are $\partial$-\emph{equivalent}.

Let $C$ be a compression body with $\Sigma=\partial_+C$ and $\Xi=\partial_-C$, and let $N$ be a manifold such that $\Sigma\subseteq\partial N$. (The orientation of $\Sigma$ is induced by the orientation of $N$ and it is opposite to the orientation induced by the orientation of $C$.) Let $\theta\colon \Sigma\to\Sigma$ and $h\colon N\to N$ be orientation-preserving homeomorphisms.

\begin{lem}\label{lema1}
The manifolds $C+_\theta N$ and $C+_{h\theta} N$ are $\Xi$-equivalent.
\end{lem}
\begin{proof}
Consider the following diagram

\begin{center}
\begin{tikzcd}
\Sigma\arrow{r}{\theta}\arrow[hook]{d}
&N \arrow{r}{h}\arrow{d}{v}
&N \arrow{r}{h^{-1}}\arrow{d}{v'}
&N \arrow{d}{v}\\
C\arrow{r}{u}\arrow[bend right]{rr}{u'}\arrow[bend right]{rrr}[swap]{u}
&C+_\theta N\arrow[dotted]{r}{w}
&C+_{h\theta} N\arrow[dotted]{r}{w'}
&C+_\theta N\\
\Xi\arrow[hook]{u}
\end{tikzcd}
\end{center}
where

\begin{center}
\begin{tikzcd}
\Sigma \arrow[r, "\theta"] \arrow[d, hook] \arrow[dr, phantom, "{\small (1)}"]
& N \arrow[d, "v"] \\
C \arrow[r, "u"]
& C+_\theta N
\end{tikzcd}
and
\begin{tikzcd}
\Sigma \arrow{r}{h\theta} \arrow[hook]{d} \arrow[dr, phantom, "(2)"]
& N \arrow{d}{v'} \\
C \arrow{r}{u'}
& C+_{h\theta} N
\end{tikzcd}
\end{center}
are pushouts (gluing).

Since (2) is a pushout, we have $u'|_\Sigma=v'h\theta$, and since (1) is a pushout there exists $w\colon C+_\theta N\to C+_{h\theta} N$ such that
\[
(\dagger)\quad u'=wu,\quad\quad (\dagger\dagger)\quad v'h=wv.
\]
Analogously, since (1) is a pushout, we have $u|_\Sigma=vh^{-1}h\theta$, and since (2) is a pushout there exists $w'\colon C+_{h\theta} N\to C+_\theta N$ such that
\[
(\ast)\quad u=w'u',\quad\quad (\ast\ast)\quad vh^{-1}=w'v'.
\]
Hence, $w'wu\stackrel{\dagger}{=}w'u'\stackrel{\ast}{=}u$ and $w'wv\stackrel{\dagger\dagger}{=}w'v'h\stackrel{\ast\ast}{=}vh^{-1}h=v$. {Finally, by the universal property of the pushout square (1), there is a unique map $g\colon C+_\theta N\to C+_\theta N$ satisfying $u=gu$ and $v=gv$. Since both $g=w'w$ and $g=\mj_{C+_\theta N}$ satisfy these conditions, uniqueness gives $w'w=\mj_{C+_\theta N}$. Analogously,  $ww'=\mj_{C+_{h\theta} N}$.} Hence, $w$ is a homeomorphism and by $(\dagger)$ we have that

\begin{center}
\begin{tikzcd}
\Xi\arrow[hook]{r}\arrow[hook]{dr}
&C+_\theta N \arrow{d}{w}\\
&C+_{h\theta} N
\end{tikzcd}
\end{center}
commutes.
\end{proof}

\begin{cor}\label{torus2}
For $C$, $\Sigma$, $N$, $\theta$ and $\Xi$ as above, where $\theta$ extends to an automorphism $h$ of $N$, we have that $C+N$ and $C+_\theta N$ are $\Xi$-equivalent. Moreover, the homeomorphism $w\colon C+N\to C+_\theta N$ underlying this $\Xi$-equivalence is the identity on $C$ and for every $x\in N$, $w(x)=h(x)$.
\end{cor}

\begin{rem}\label{isotopy gluing}
If $\theta,\theta'\colon\Sigma\to \Sigma$ are isotopic, i.e.\ if they denote the same element of the mapping class group of $\Sigma$, then $C+_\theta N$ and $C+_{\theta'} N$ are $\Xi$-equivalent.
\end{rem}

\begin{definition}
Let $l$ be a simple closed curve in $\Sigma$. Consider an annulus in $\Sigma$ with $l$ as a component of its boundary, {and denote its other boundary component by $l'$}. We say that $\theta\colon\Sigma\to\Sigma$ is a \emph{Dehn twist} when it is the identity outside the annulus and it is defined by cutting $\Sigma$ along $l$ and rotating for $360^\circ$ the copy of $l$ bounding the annulus, and then gluing it with the other copy back again. See Figure~\ref{Dehn} where every radius of the annulus is mapped by $\theta$ and $\theta^{-1}$ into a spiral.
\end{definition}

\begin{figure}[h!h!h!]
    \centerline{\includegraphics[width=1\textwidth]{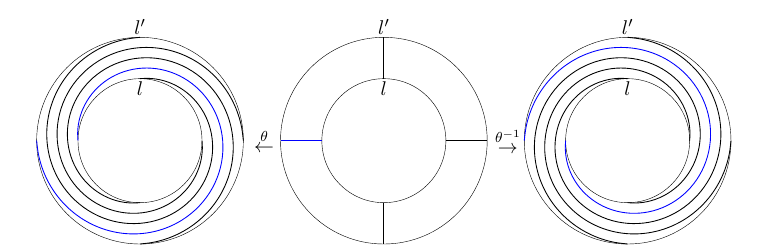} } \caption{Dehn twist and its inverse. }\label{Dehn}
\end{figure}

\begin{thm}{ \cite[End of \S10/ 3)/ c)]{D38}}, \cite[Theorem~1]{L62}\label{Lickorish}
Every { orientation-preserving} homeomorphism of $\Sigma$ is isotopic to a composition of Dehn twists.
\end{thm}

Consider the simplest case when $\theta$ is a Dehn twist along $l\subseteq\Sigma$ and $M$ is $C+_\theta H$. Let a blade of the shape given at the left-hand side of Figure~\ref{hirurgija1} be immersed in $H$. Imagine that $A$ slides along $l$, while $AB$ is perpendicular to $\Sigma$ and $l$ is perpendicular to the plane of the blade (see Figure~\ref{hirurgija1}; the line $DC$ intersects $\Sigma$ in { the other component, denoted by  $l'$, of the boundary of the chosen annulus on $\Sigma$}). Note that $l$ could be knotted, which makes the picture more involved.
\begin{figure}
    \centerline{\includegraphics[width=0.9\textwidth]{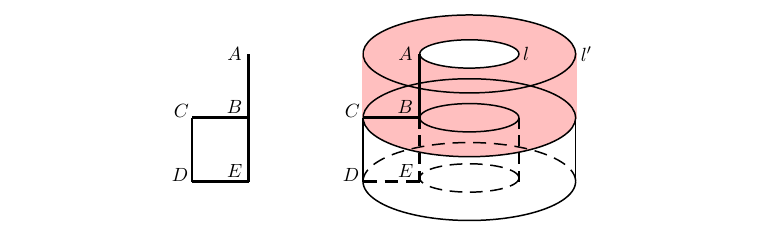}} \caption{}\label{hirurgija1}
\end{figure}

Let $N$ be $H-\overset{\circ}{T}$ where $T$ is the solid torus (with corners) cut out from $H$ by the square $BCDE$ of the blade. Let $h\colon N\to N$ be the identity everywhere but in the region between $T$ and $\Sigma$ (highlighted in Figure~\ref{hirurgija1}). Let $h$ be defined at each level $t$ between $A$ and $B$ of that region, as the Dehn twist $\theta$. This makes $h$ an automorphism of $N$ whose restriction to $\Sigma$~is~$\theta$.

Let us denote the manifold $C\cup H$, which is $S^3$ with some unlinked and unknotted handlebodies removed, by $S^3_-$. Note that the solid torus $T$ could be knotted and linked with the other handlebodies mentioned above. We have that $C\cup N$ is $S^3_- - \overset{\circ}{T}$ and by Corollary~\ref{torus2}, for $\Xi$ being $\partial_- C$, it is $\Xi$-equivalent to $C+_\theta N$, which is $M-\overset{\circ}{T}$.

Since $(M-\overset{\circ}{T})\cup T=M$ and $M-\overset{\circ}{T}$ and $S^3_- - \overset{\circ}{T}$ are homeomorphic, the manifold $M$ could be obtained from $S^3_- - \overset{\circ}{T}$ by sewing back the solid torus $T$ by an appropriate gluing homeomorphism. The following lemma formalises this procedure and describes this gluing homeomorphism.

\begin{lem}\label{lema2}
Let $K$, $K'$ and $L$ be manifolds and let $\Delta$ and $\Xi$ be disjoint closed surfaces such that $\Delta\cup\Xi=\partial K=\partial K'$, and $\Delta=\partial L$. If $K$ and $K'$ are $\Xi$-equivalent via $w$, and $\varphi\colon\Delta\to\Delta$ is the restriction $w^{-1}|_\Delta$, then  $K\prescript{}{\varphi}{+}L$ and $K'+L$ are $\Xi$-equivalent.
\end{lem}

\begin{proof}
Consider the following diagram,
\begin{center}
\begin{tikzcd}
&\Delta\arrow[hook]{r}\arrow{d}{\varphi}
&L \arrow{d}{v}\arrow[bend left=80]{dd}[swap]{v'}\arrow[bend left=80]{ddd}{v}\\
\Xi\arrow[hook]{r}\arrow[hook, bend right]{dr}
&K\arrow{r}{u}\arrow{d}{w}
&K\prescript{}{\varphi}{+}L\arrow[dotted]{d}{h}\\
&K'\arrow{r}{u'}\arrow{d}{w^{-1}}
&K'+L\arrow[dotted]{d}{h'}\\
&K\arrow{r}{u}
&K\prescript{}{\varphi}{+}L
\end{tikzcd}
\end{center}
where the two diagrams:
\begin{center}
\begin{tikzcd}
\Delta\arrow[hook]{r}\arrow{d}{\varphi}
&L \arrow{d}{v}\\
K\arrow{r}{u}
&K\prescript{}{\varphi}{+}L
\end{tikzcd}
and
\begin{tikzcd}
\Delta\arrow[hook]{r}\arrow[hook]{d}
&L \arrow{d}{v'}\\
K'\arrow{r}{u'}
&K'+L
\end{tikzcd}
\end{center}
are pushouts. We have that $h'h=\mj_{K\prescript{}{\varphi}{+}L}$ as in the proof of Lemma~\ref{lema1} and $hh'=\mj_{K'+L}$ is proved analogously. That $h$ underlies the $\Xi$-equivalence between $K\prescript{}{\varphi}{+}L$ and $K'+L$ follows from the commutativity of the triangle and the square in the middle of the initial diagram.
\end{proof}

Let now $K$, $K'$ and $L$ be respectively $S^3_--\overset{\circ}{T}$, $M-\overset{\circ}{T}$ and $T$, and let $\Xi=\partial_-C$. Denote by $w$ the homeomorphism underlying the $\Xi$-equivalence between $S^3_--\overset{\circ}{T}$ and $M-\overset{\circ}{T}$, which is guaranteed by Corollary~\ref{torus2}. According to Lemma~\ref{lema2}, we have that $(S^3_--\overset{\circ}{T})\prescript{}{\varphi}{+} T$ and $(M-\overset{\circ}{T})\cup T$, which is $M$, are $\Xi$-equivalent and the homeomorphism for sewing the solid torus $T$ back to $S^3_--\overset{\circ}{T}$ is the identity on its three sides containing respectively $CD$, $DE$, $BE$ (see Figure~\ref{hirurgija1}) and it is the inverse of the Dehn twist $\theta$ (see Figure~\ref{Dehn}) at the level containing $BC$ (see the last sentence of Corollary~\ref{torus2} and the last sentence in the second paragraph after Theorem~\ref{Lickorish}).

Consider the meridian $m$ consisting of the line segments $BC$, $CD$, $DE$ and $EB$ of the solid torus $T$, which is marked in blue at the left-hand side of Figure~\ref{meridian}. Reasoning as above, we conclude that this meridian is mapped by $\varphi$ into the curve $\mu$ illustrated at the right-hand side of Figure~\ref{meridian}.

\begin{figure}[h!]
    \centerline{\includegraphics[width=0.9\textwidth]{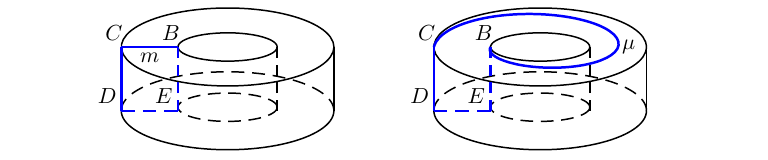} } \caption{Meridian and its image}\label{meridian}
\end{figure}

By \ref{sim}, we know that the solid torus $T$ is sewed back to $S^3_--\overset{\circ}{T}$ in order to obtain $M$ so that its meridian $m$ is sewed back to the curve $\mu$. We call $\mu$ the \emph{attaching curve}.

\begin{rem}\label{sewing meridian}
Sewing back a solid torus to a manifold with such a torus removed depends only on sewing back its meridional disk. This is because one can perform this action in two steps: first sew back a regular neighbourhood of the meridional disk, and then sew back the rest, which is a ball. The result of the second step does not depend on the chosen gluing homeomorphism of the boundary sphere, since all the orientation-preserving homeomorphisms of $S^2$ are isotopic, and one could apply Remark~\ref{isotopy gluing} or just use the fact that all homeomorphisms of $S^2$ extend over $D^3$.
\end{rem}

The solid torus $T$ in $S^3_-$ is completely determined by its core $\gamma$ (as the closure of a tubular neighbourhood of $\gamma$), hence the transformation of $S^3_-$ into $M$ is completely determined by the link consisting of $\gamma$ and the attaching curve $\mu$. Assume that the two components of this link are \emph{codirected} (given two orientation vectors on $\gamma$ and $\mu$ at points of the same meridional disk, their scalar product must be positive). Since $\gamma$ and $\mu$ are homologous in $T$, the core $\gamma$ and the linking number ${\rm lk}(\mu,\gamma)$ are sufficient to describe the transformation of $S^3_-$ into $M$. This is a \emph{framed link} description of $M$. In our case, the framed link description is given by one circle corresponding to $\gamma$ with framing $1$.

\begin{figure}[H]
    \centerline{\includegraphics[width=0.9\textwidth]{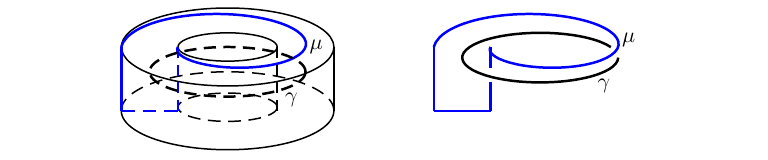} } \caption{Surgery data}\label{surgery data}
\end{figure}

In a more involved case, when $\theta$ in $C+_\theta H$ is a composition $\delta_n\circ\ldots\circ\delta_1$, for $n>1$, of Dehn twists, we proceed as follows. Start with a blade $ABCDE$ whose \emph{edges} $AB$, $BC$, $CD$, $DE$, $EB$ are all of length $\varepsilon$. Let $A$ slide along the curve in $\Sigma$ corresponding to $\delta_n$ as before. Let $T_n$ be the solid torus removed in this step. By Lemma~\ref{lema1}, we conclude that $C+_{\delta_{n-1}\circ\ldots\circ\delta_1}(H-\overset{\circ}{T_n})$ and $C+_{\delta_n\circ\ldots\circ\delta_1}(H-\overset{\circ}{T_n})$, which is $M-\overset{\circ}{T_n}$, are $\Xi$-equivalent.

Repeat this procedure with the blade whose sides are of length $\varepsilon/3$ and $A$ slides along the curve corresponding to $\delta_{n-1}$. Let $T_{n-1}$ be the solid torus removed in this step. Again, by Lemma~\ref{lema1} we have that $C+_{\delta_{n-2}\circ\ldots\circ\delta_1}(H-(\overset{\circ}{T_n}\cup \overset{\circ\mbox{\hspace{1.2em}}}{T_{n-1}}))$ and $M-(\overset{\circ}{T_n}\cup \overset{\circ\mbox{\hspace{1.2em}}}{T_{n-1}})$ are $\Xi$-equivalent.

After $n$ such steps we have that $S^3_--(\overset{\circ}{T_n}\cup\ldots\cup \overset{\circ}{T_1})$ and $M-(\overset{\circ}{T_n}\cup\ldots\cup \overset{\circ}{T_1})$ are $\Xi$-equivalent. By applying Lemma~\ref{lema2}, for a gluing homeomorphism $\varphi$ calculated at each component $\partial T_i$ as in the previous case, we have the following result.

\begin{prop}\label{surgery language}
Every manifold $M=C+_\theta H$ is $\partial$-equivalent to $(S^3_--(\overset{\circ}{T_n}\cup\ldots\cup \overset{\circ}{T_1}))\prescript{}{\varphi}{+} (T_n\cup\ldots\cup T_1)$, for some solid tori $T_1,\ldots T_n$ in $H$ and some gluing homeomorphism $\varphi$.
\end{prop}

A diagram of our graphical language is placed in $S^3$, and it consists of a finite collection of unlinked and unknotted wedges of circles and a framed link. (A wedge of zero circles is just a point.) We present these diagrams as planar (with undercrossings and overcrossings) like the one given in Figure~\ref{dijagram1}, and we call them $\mathbb{Z}$-\emph{diagrams}.

One does not distinguish two $\mathbb{Z}$-diagrams differing by Reidemeister moves applied to framed links (the wedges of circles must remain fixed since we are interested in $\partial$-equivalence of manifolds). We call such $\mathbb{Z}$-diagrams \emph{Reidemeister equivalent}. This is because we consider presentations in $S^3$ up to isotopy that fixes the wedges of circles.

The \emph{interpretation} of a $\mathbb{Z}$-diagram as a manifold is as at the beginning of the section. Directly from Proposition~\ref{surgery language} we have the following.

\begin{cor}
{Every manifold is homeomorphic to the interpretation of a $\mathbb{Z}$-diagram.}
\end{cor}

\section{The calculus}\label{calculus}

Our next step is to introduce a calculus on $\mathbb{Z}$-diagrams consisting of the following three moves (see Figure~\ref{dijagram2}). This calculus has its origins in the work of Kirby, \cite{K78}, Fenn and Rourke, \cite{FR79} and Roberts, \cite{R97}. Since the labels for link components are integers, it is called \emph{integral surgery calculus}.

In what follows, we shall refer to the threads coloured in red, i.e., that belong to the wedges of circles, as \emph{wedge threads}. The change in framing of a link component $C_i$ is indicated in Moves ($-1$) and ($1$). Note that we assume that the circles with framing $-1$ and $1$ at the right-hand sides of Moves~($-1$) and ($1$) vanish when there are no wedge threads passing through~$C$.

\begin{figure}[h!]
    \centerline{\includegraphics[width=1\textwidth]{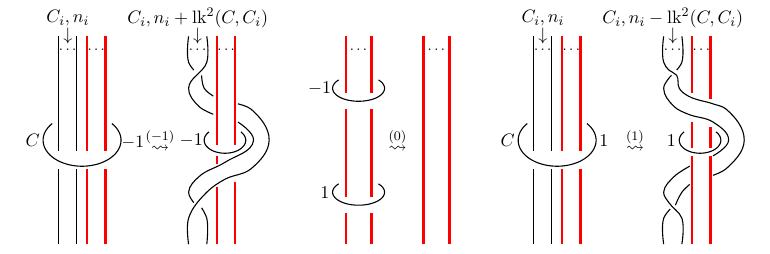}} \caption{Moves ($-1$), ($0$) and ($1$)}\label{dijagram2}
\end{figure}

Although, throughout the paper the moves are directed, we always consider the symmetric closure of the underlying relation on diagrams. We say that two $\mathbb{Z}$-diagrams are $\mathbb{Z}$-\emph{equivalent} if there is a finite sequence of Moves ($-1$), ($0$) and ($1$) and their inverses transforming one into the other. The following result proves that the integral surgery calculus is complete.

\begin{thm}\label{integral calculus}
Two $\mathbb{Z}$-diagrams with identical wedges of circles denote two $\partial$-equivalent manifolds iff they are $\mathbb{Z}$-equivalent.
\end{thm}
\begin{proof}
$(\Leftarrow)$ Let $M$ be a manifold denoted by a $\mathbb{Z}$-diagram containing a fragment of the form of the left-hand side of Move ($-1$). In order to have a clear picture we start with a concrete example of this fragment given in Figure~\ref{dijagram5}
\begin{figure}[h!]
    \centerline{\includegraphics[width=.3\textwidth]{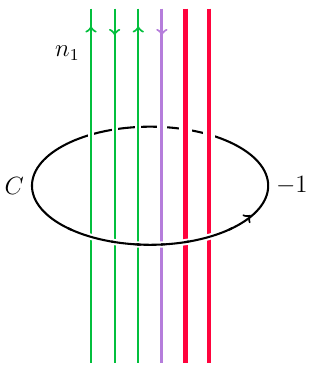} } \caption{}\label{dijagram5}
\end{figure}
where the same colour denotes the same link component (also, the link components are oriented, which serves for counting linking numbers and the choice of orientation does not effect the resulting framing).  If the linking number of an oriented thread with $C$ is equal to $1$, then we call it \emph{outgoing} thread. Otherwise, it is \emph{incoming} thread. The first and the third green thread in Figure~\ref{dijagram5} are outgoing, while the second green thread and the purple thread are incoming threads.

The component $C$ of the link describes an unknotted solid torus $T$ in $S^3$ whose core is $C$. By removing $T$ and sewing back its meridian along a curve $C'$ having ${\rm lk}(C',C)=-1$ we perform one step of the surgery leading from $S^3_-$ to $M$.

Consider the cylinder $W=D^2\times I\subseteq S^3-\overset{\circ}{T}$, which is attached to $\partial T$ along $S^1\times I$ . The \emph{twist homeomorphism} of $W$,
\[
(x,t)\mapsto(xe^{2\pi it},t)
\]
extended by the identity of $S^3-(\overset{\circ}{T}\cup W)$ defines an automorphism of ${S^3-\overset{\circ}{T}}$.

In other words, imagine that $W$ makes a well in $S^3$. Form a cut at the top level of this well (highlighted in Figure~\ref{twist}) and make a twist so that this level is rotated for $360^\circ$ counterclockwise and the bottom level of the well is fixed. The spiral in the left-hand side picture is mapped into the vertical line segment connecting the same endpoints, which is a part of the blue meridian in the right-hand side picture.

\begin{figure}[h!]
    \centerline{\includegraphics[width=0.9\textwidth]{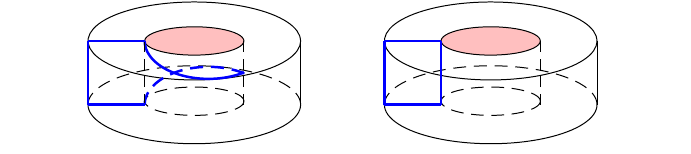} } \caption{Twist homeomorphism}\label{twist}
\end{figure}

This means that after performing all the other surgeries, the solid torus $T$ should be sewed back so that its meridian is identified with itself. Hence, this surgery is trivial and could be omitted. It remains to see what has happened inside the well. Figure~\ref{dijagram6}a) illustrates the change in this region, and Figure~\ref{dijagram6}b) adds an isotopy that separates
the upper part where the wedge threads are fixed and the lower part where they are twisted.
\begin{figure}[h!]
    \centerline{\includegraphics[width=.6\textwidth]{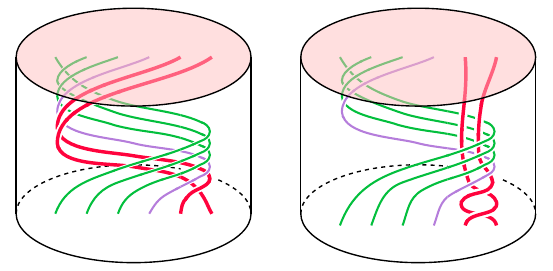}
    } \caption{a) and b)}\label{dijagram6}
\end{figure}

Take a look for a moment at Figure~\ref{dijagram22}, where the desired result of Move {($-1$)} applied to Figure~\ref{dijagram5} is illustrated.
\begin{figure}[h!]
    \centerline{\includegraphics[width=.4\textwidth]{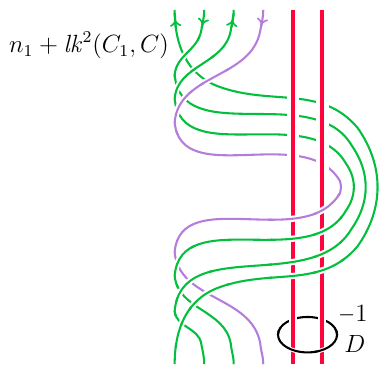}  } \caption{}\label{dijagram22}
\end{figure}
After performing the twist homeomorphism with respect to the component $D$ in Figure~\ref{dijagram22}, and by neglecting the change of framing for other components of the link, one can easily see that it corresponds to Figure~\ref{dijagram6}b). It remains to check the change of framing tied to the components passing through $C$.

Let us count the new framing of the green component of the link whose initial framing was $n_1$. Denote this component by $C_1$. Assume that  before applying the twist homeomorphism, the black line segments in Figure~\ref{dijagram8}a) (codirected with the green line segments) belong to the attaching curve tied to $C_1$. The situation after the twist homeomorphism is illustrated in Figure~\ref{dijagram8}b).
\begin{figure}[H]
    \centerline{\includegraphics[width=.6\textwidth]{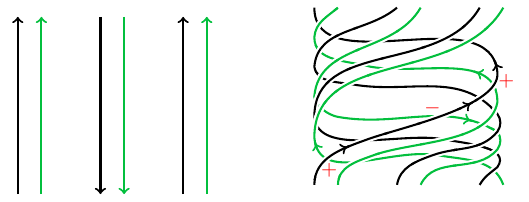}
    } \caption{a) and b)}\label{dijagram8}
\end{figure}

The linking number between the black and the green component (which contributes to the framing of $C_1$) has changed by
\[
O\cdot(O-I)+ I\cdot(I-O)=(O-I)^2
\]
where $O$ is the number of outgoing threads of $C_1$ while $I$ is the number of incoming threads of $C_1$. In Figure~\ref{dijagram8}b), the leftmost black (outgoing) thread overcrosses three green threads and the linking number increases by $1$ in the case when it overcrosses the outgoing green thread, and decreases by one otherwise. Since $O-I={\rm lk}(C_1,C)$,  the new framing of $C_1$ is $n_1+{\rm lk}^2(C_1,C)$.

This finishes the justification of Move ($-1$). We proceed analogously in the case of Move ($1$), and the case of Move ($0$) is trivial since the corresponding twist homeomorphisms cancel each other.

\vspace{2ex}
\noindent$(\Rightarrow)$ For this direction we rely on \cite[Theorem~1]{R97}. Our notion of $\partial$-equivalent manifolds corresponds to the boundary condition mentioned in the third sentence of the second paragraph in \cite{R97}. The diagrammatic calculus for manifolds would be different if we classify them only by homeomorphism not respecting boundaries as it was shown in example illustrated in Figure~\ref{d100}.

Theorem~1 from \cite{R97} claims that the moves given in Figure~\ref{dijagram10} suffice for showing that two diagrams denote $\partial$-equivalent manifolds. Move R1 is supported in a ball non-intersecting any other wedge or link component of the diagram, while R2 and R3 are supported in any embedded image (with the same non-intersecting property) of genus-2 handlebody and solid torus, respectively.
\begin{figure}[H]
    \centerline{\includegraphics[width=.7\textwidth]{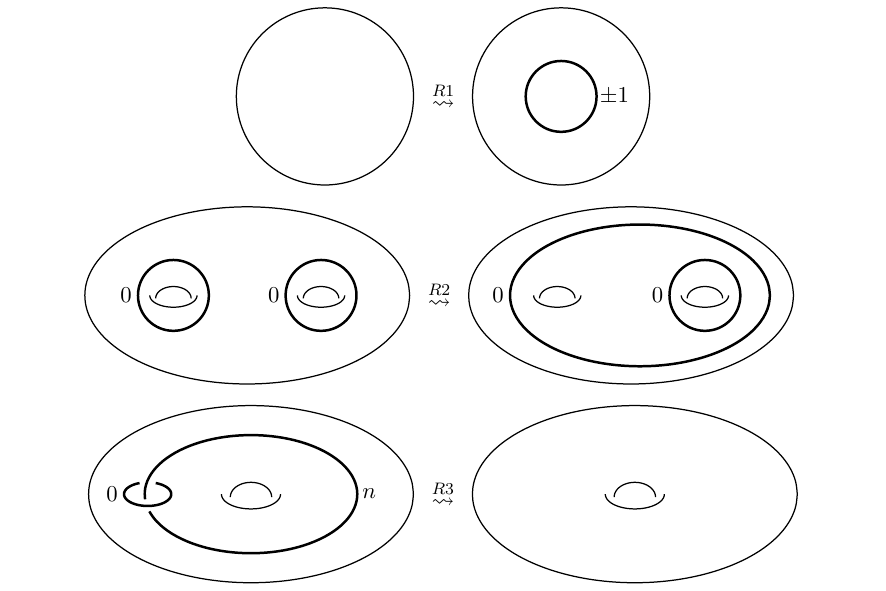}}
    \caption{Roberts' moves }\label{dijagram10}
\end{figure}

Move R{1} is trivially covered by Moves ($-1$) and ($1$).

For Move~R{2}, we first note that the ``unknotting moves'' illustrated at the top of Figure~\ref{dijagram691}, as well as the ``change-of-framing'' moves illustrated at the bottom of Figure~\ref{dijagram691}, are derivable from Moves~($-1$), ($0$) and ($1$).

\begin{figure}[H]
    \centerline{\includegraphics[width=0.8\textwidth]{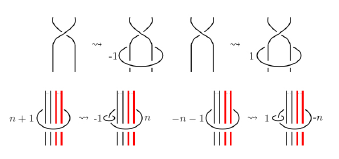}
    } \caption{{ The unknotting moves (top) and the  change-of-framing moves (bottom)}}\label{dijagram691}
\end{figure}
The ``unknotting moves''  assume that two  threads are \emph{not} co-directed and belong to the same (arbitrary oriented) link component of the surgery. Hence, the linking number of this component with the unknot whose framing is $-1$ or $1$ at the right-hand side of the moves is $0$ and these moves are derived by applying the inverse of Move ($-1$) (resp. Move ($1$)) to the pair of parallel strands,  followed by a single Type II Reidemeister move.
The ``change-of-framing''  moves are derived  by applying the inverse of Move ($-1$) (resp. Move ($1$)) to a single strand.

\begin{figure}[H]
    \centerline{\includegraphics[width=0.4\textwidth]{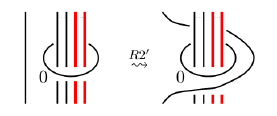}}
    \caption{{ Move  R$2'$}}\label{dijagram69}
\end{figure}

Now,  with the help of the unknotting moves and the change-of-framing moves, it suffices to show that Move R$2'$ from Figure~\ref{dijagram69} is derivable from Moves ($-1$), ($0$) and ($1$).

\begin{figure}[H]
 \centerline{\includegraphics[width=1\textwidth]{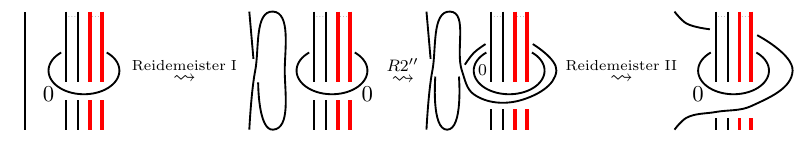}
    } \caption{{ Derivation of Move  R$2'$  from Move R$2''$}}\label{dijagram101}
\end{figure}

Indeed, since the Move R2 is supported in any embedded image of the  genus-two handlebody in $S^3$, the $0$-framed unknots are actually arbitrarily knotted inside $S^3$. Starting with the left-hand side of Move R2, we first apply the unknotting move  to untangle the knotted loop into a standard unknot. We then use the change-of-framing move to set the resulting framing of this unknot to $0$.

As shown in  Figure~\ref{dijagram101}, Move R$2'$ itself follows from the move denoted there by R2'', whose derivation is given in Figure~\ref{dijagram23}.
The second row in Figure~\ref{dijagram23} features  two instances of Move ($-1$): the first one acts on the left and is applied   to the outer $(-1)$-framed circle, and the second one acts on the right and is applied to inner $(-1)$-framed circle.  Note that during this derivation, we do not register the change of framing for all threads. It depends on whether the corresponding link component passes through the circle with framing $0$, or not. However, it is easy to see that the terminal value of each framing is the same as the initial.

\begin{figure}[H]
    \centerline{\includegraphics[width=1.1\textwidth]{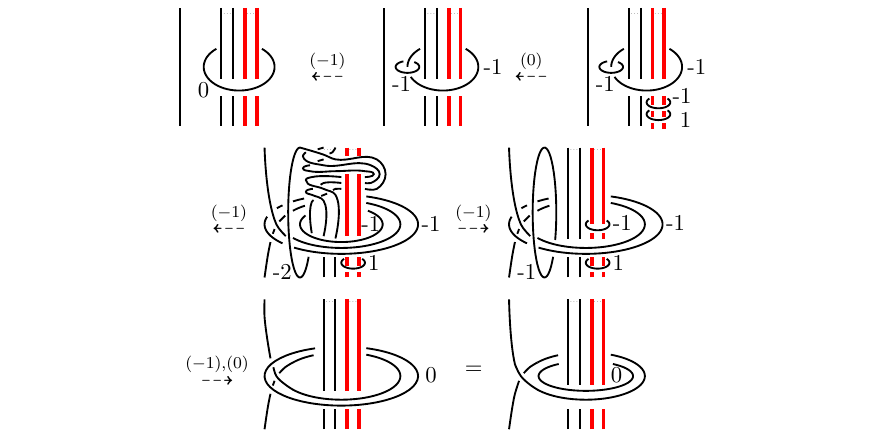}
    } \caption{Derivation of R2'' (some steps require isotopy) }\label{dijagram23}
\end{figure}

Finally, for Move~R{3}, we start with reducing it to the case $n=-1$. This is done by iterating the procedure illustrated in Figure~\ref{dijagram13}.
\begin{figure}[H]
    \centerline{\includegraphics[width=1\textwidth]{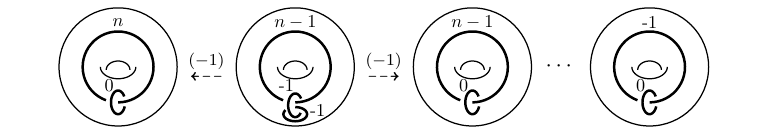}}
    \caption{{ Reducing the move  R3 to the case $n=-1$ }}\label{dijagram13}.
\end{figure}
Each step in Figure~\ref{dijagram13}, which is devoted to the case $n\geq 0$, consists of a pair of applications of Move ($-1$), performed on the same intermediate configuration (appearing at an even position in the sequence): first the inverse Move ($-1$), acting to the left, followed by the direct Move ($-1$), acting to the right. Starting from the second configuration above, the inverse Move ($-1$) is applied to the small $(-1)$-framed circle linking   the main component, whose framing is $n-1$, and the second  auxiliary small circle, also with framing  $-1$, producing the configuration on the left-hand side of Move R3. The subsequent direct Move ($-1$) is then applied to the auxiliary small circle with framing  $-1$, which cancels the effect of the modification on the small $0$-framed circle caused by the first move. Therefore, the composite transformation has the effect of decreasing the framing of the main component by $1$, while leaving the framing of the small circle unchanged.  Repeating this procedure $n+1$ times reduces the framing to $-1$.

Again, by applying the unknotting move, the derivation of the reduced case of R3 from Moves ($-1$), ($0$) and ($1$) boils down to the steps illustrated in Figure~\ref{dijagram14}. This finishes the proof of Theorem~\ref{integral calculus}.\end{proof}
\begin{figure}[H]
    \centerline{\includegraphics[width=0.8\textwidth]{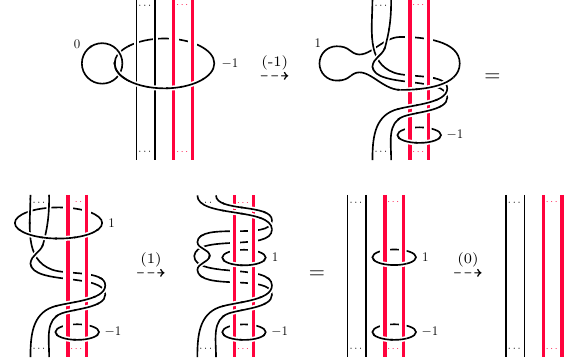} }
    \caption{Derivation of R3}\label{dijagram14}
\end{figure}

\subsection{A finite list of local moves}\label{local}

It is shown in \cite{M12} that the Fenn and Rourke list of local moves for closed, oriented 3-manifolds can be reduced to a finite list given in \cite[Fig.~3]{M12}.
\begin{figure}[H] \centerline{\includegraphics[width=0.9\textwidth]{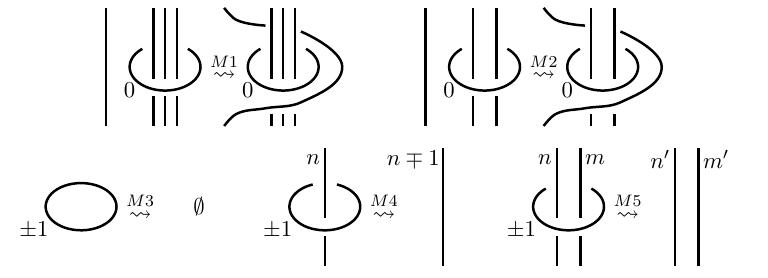}}
    \caption{Martelli's local moves M1-M5}\label{dijagram53}
\end{figure}

We prove that there is a similar finite list of local moves sufficient for $\partial$-equivalence of manifolds. In addition to Martelli's moves, which are presented in Figure~\ref{dijagram53}, we consider the local moves with a wedge thread involved shown in Figure~\ref{dijagram54}. (The framings $n'$ and $m'$ below are computed in the following way: if both threads passing through the circle $C$ with framing $\pm1$ belong to the same component $C_1$, then the new framing of this component is $n\mp lk^2(C,C_1)$; otherwise, $n'=n\mp1$ and $m'=m\mp1$.)

\begin{thm}\label{oznaka}
Two $\mathbb{Z}$-diagrams are $\mathbb{Z}$-equivalent iff they are connected by a sequence of moves M1-W4 and their inverses.
\end{thm}
\begin{figure}[H]
    \centerline{\includegraphics[width=0.9\textwidth]{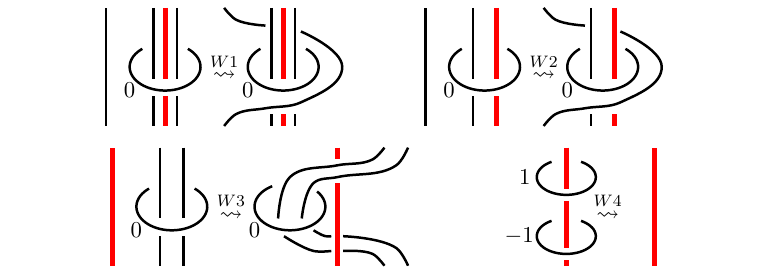}}
    \caption{Local moves with wedges W1-W4}\label{dijagram54}
\end{figure}
\begin{proof}
First we show that Moves M1-W4 can be derived from Moves ($-1$), ($0$) and ($1$). Moves M1, M2, W1 and W2 are derived as Move~R2 in the proof of Theorem~\ref{integral calculus}. Move W3 is derived analogously (see Figure~\ref{dijagram57}).

\begin{figure}[h!h!h!]
    \centerline{\includegraphics[width=0.9\textwidth]{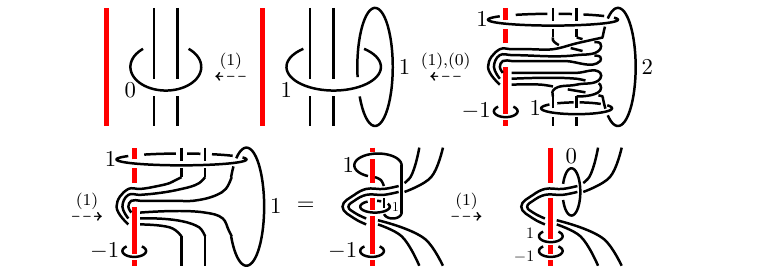}}
    \caption{Derivation of W3 (some steps require isotopy)}\label{dijagram57}
\end{figure}

Moves M3, M4 and M5 are instances of Moves ($-1$) and ($1$). Note that Move W4 is the special case of Move ($0$) where only a single wedge thread passes through two circles with framing $-1$ and $1$.

Next we show that Moves ($-1$), ($0$) and ($1$) are derivable from M1-W4. For this we follow \cite{M12} and start with useful local transformations $(\dagger\emptyset)$ and $(\dagger)$ shown in Figures~\ref{dijagram64a}-\ref{dijagram64}.
\begin{figure}[H]
    \centerline{\includegraphics[width=0.9\textwidth]{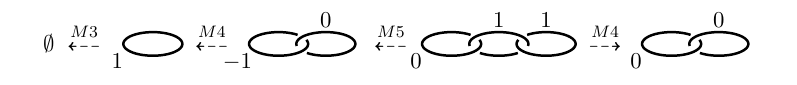} }
    \caption{Transformation $(\dagger\emptyset)$}\label{dijagram64a}
\end{figure}
\vspace{-0.5cm}
\begin{figure}[H]
    \centerline{\includegraphics[width=0.9\textwidth]{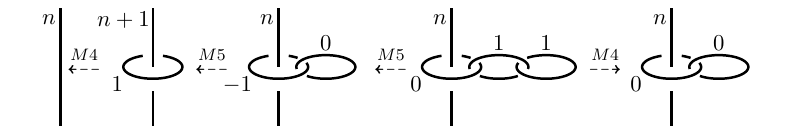} }
    \caption{Transformation $(\dagger)$}\label{dijagram64}
\end{figure}
The latter transformation, together with M2 or W3 produces the transformation $(\dagger\dagger)$ from Figure~\ref{dijagram65} where the blue thread could be either part of a wedge or of a link component. (The fragment prepared for either M2 or W3 is highlighted in the middle.)
\begin{figure}[h!h!h!]
    \centerline{\includegraphics[width=.6\textwidth]{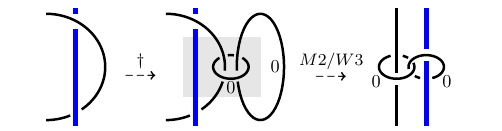} }
    \caption{Transformation $(\dagger\dagger)$}\label{dijagram65}
\end{figure}

For Move ($-1$) we exhibit the case of four threads (two from wedges and two from link components) passing through a circle with framing $-1$ (see Figure~\ref{dijagram66} { in which the first line is obtained by introducing six linked circles with framing $0$, i.e. by three applications of $(\dagger)$ followed by several applications of either M2 or W3}).
\begin{figure}[H]
    \centerline{\includegraphics[width=0.9\textwidth]{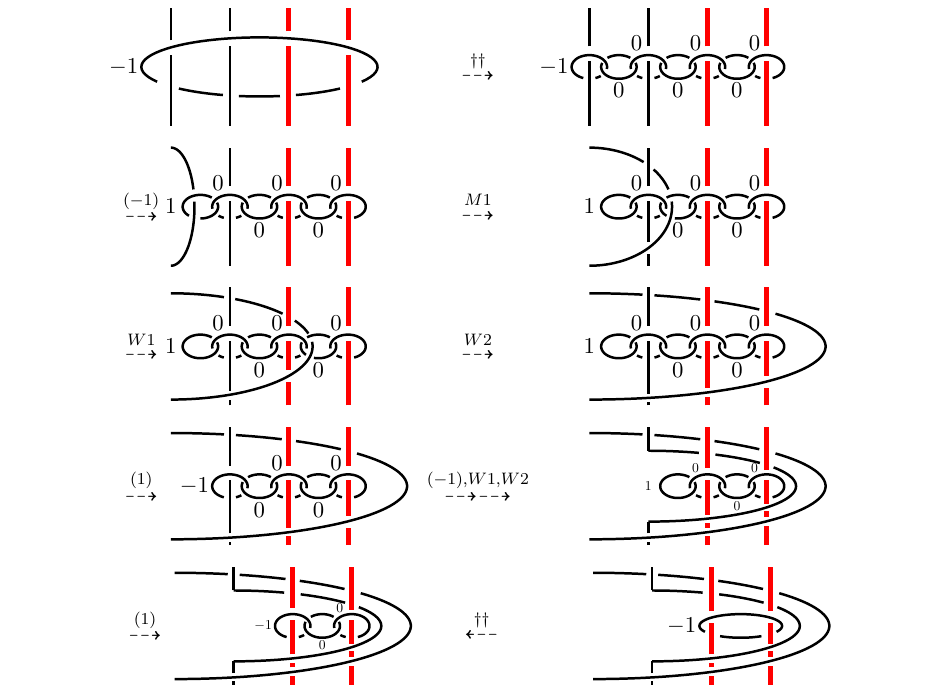} }
    \caption{Derivation of Move (-1)}\label{dijagram66}
\end{figure}
\noindent This suffices for catching the iterative procedure solving the general case. Move ($1$) is derived analogously. (Note that in the case of zero threads passing through a circle with framing $\pm 1$ we apply M3.)

The local moves W5 and W6 given in Figure~\ref{dijagram68} are derivable from Move ($1$) and W4 in the same manner as shown in Figure~\ref{dijagram57}. (Note that Move~($0$) is used in Figure~\ref{dijagram57} just in the form of W4.) Thus, once we have derived Moves ($-1$) and ($1$), the moves W5 and W6 are available for derivation of Move ($0$).
\begin{figure}[H]
    \centerline{\includegraphics[width=.8\textwidth]{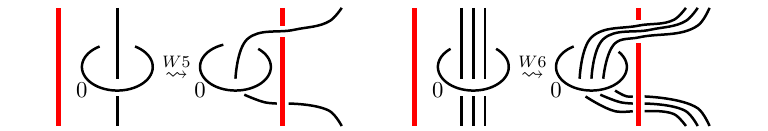}}
    \caption{Local moves W5 and W6}\label{dijagram68}
\end{figure}
A derivation of Move ($0$) from the local moves proceeds as in Figure~\ref{dijagram67}. (Note that in the case of $n$ wedge threads, W5 and W6 are applied $n-1$ times in this derivation.) This finishes the proof of Theorem \ref{oznaka}.
\end{proof}

\begin{figure}[H]
    \centerline{\includegraphics[width=1\textwidth]{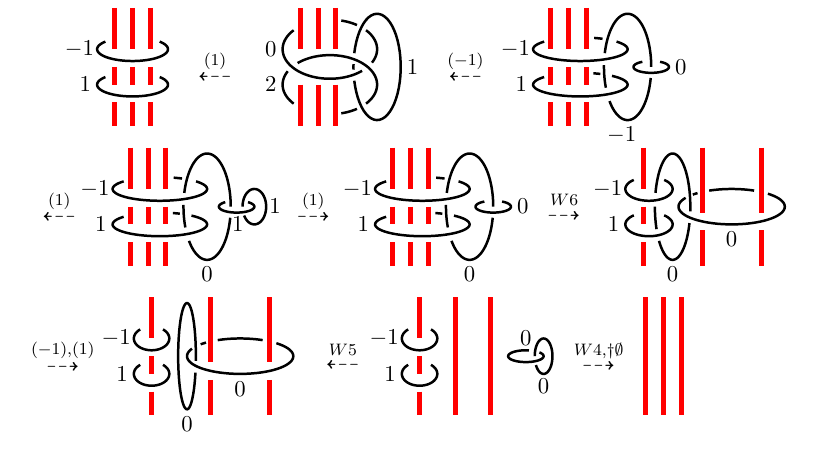}}
    \caption{Derivation of Move ($0$)}\label{dijagram67}
\end{figure}

\subsection{The rational calculus}

Besides the integral surgery calculus, one may consider its extension taking into account rational framing of link components. Namely, the presentation of manifolds in Section~\ref{diagrammatics} required just homologous pairs of curves (one being the core of a solid torus and the other the attaching curve on its boundary). The extension of the integral language is obtained by omitting this requirement and allowing meridian to be sewed back along any closed, essential curve with no self-intersections that winds around the boundary torus $p$ times along the meridian and $q$ times along the parallel. \emph{Essential} means that by removing this curve the rest of the torus remains connected. In this case the integers $p$ and $q$ are co-prime, or one of them is $0$ and the other is $1$ or $-1$ (see e.g. \cite[Proposition~14.1]{PS}). Moreover,  pairs $(-p,q)$ and $(p,-q)$ determine the same result of surgery.

The rational surgery calculus for closed manifolds is introduced by Rolfsen, \cite{R76}, independently of Kirby's work. He proved, by relying on Kirby's result, the completeness for this calculus in \cite{R84}. We follow that proof in showing the completeness for our rational surgery calculus. The main difference is that we formulate our moves in terms of diagrams with numerical framings, while Rolfsen's moves: \emph{Homeomorphism}, \emph{Twist} and \emph{Trivial insertion} were formulated in terms of pairs (torus, attaching curve) embedded in $S^3$.

The words in the new language are diagrams as in the language for integral calculus, save that we allow now framing to be a rational number $p/q$ including { $\infty$} (i.e.\ $\pm 1/0$). We call such diagrams the $\mathbb{Q}$-\emph{diagrams}. The framing { $\infty$} denotes the trivial surgery (note that this notation does not agree with the notation given in \cite{H07}). Besides Moves ($-1$) and ($1$), in rational surgery calculus we have Move~($2$) (see Figure~\ref{dijagram32}) and Move~($\infty$). The latter move inserts (deletes) a knot projection with framing  { $\infty$}  anywhere in the diagram.

\begin{figure}[h!]
    \centerline{\includegraphics[width=1\textwidth]{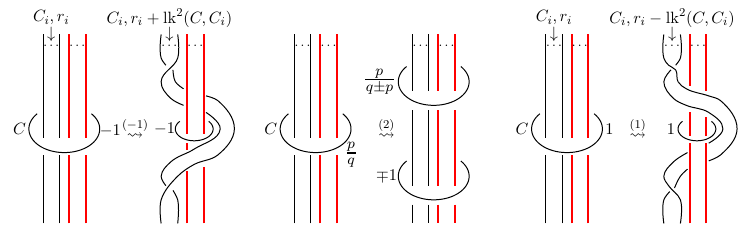}
    } \caption{Moves ($-1$), ($2$) and ($1$) for rational calculus}\label{dijagram32}
\end{figure}

We say that two $\mathbb{Q}$-diagrams are $\mathbb{Q}$-\emph{equivalent} if there is a finite sequence of Moves ($-1$), ($1$), ($2$) and ($\infty$) transforming one into the other.

\begin{prop}\label{rational calculus}
Two $\mathbb{Q}$-diagrams with identical wedges of circles denote two $\partial$-equivalent manifolds iff they are $\mathbb{Q}$-equivalent.
\end{prop}
\begin{proof}
$(\Leftarrow)$ In order to justify Move~($-1$), we proceed as in the proof of Theorem~\ref{integral calculus}, save that in Figure~\ref{dijagram8}, if the framing of the green component (denoted again by $C_1$) is $r_1=p_1/q_1$, then we have $q_1$ black threads codirected with each green thread. After performing the twist homeomorphism, the attaching curve still winds $q_1$ times around the boundary torus along the parallel. It remains to calculate the new linking number between the black and the green component, i.e.\ the change in winding of the attaching curve around the boundary torus along the meridian. For this we use the following formula
\[
q_1\cdot O\cdot(O-I)+ q_1\cdot I\cdot(I-O)=q_1\cdot (O-I)^2
\]
where $O$ is the number of outgoing threads of $C_1$ while $I$ is the number of incoming threads of $C_1$. Again, $O-I={\rm lk}(C_1,C)$, and the new framing of $C_1$ is $(p_1+q_1\cdot{\rm lk}^2(C_1,C))/q_1=r_1+{\rm lk}^2(C_1,C)$. Move~($1$) is justified analogously.

For Move~($2$), we start with the twist homeomorphism corresponding to the curve $C$ whose framing is $p/q$. The boundary torus { of its tubular neighbourhood} is presented in Figure~\ref{dijagram35} (we assume the standard identification of the opposite sides of rectangles). Before performing the twist homeomorphism, the attaching curve on this torus (for $p=3$ and $q=2$) could be taken as the one in Figure~\ref{dijagram35}a), where the upper part, separated with a dotted line, bounds the well in Figure~\ref{twist}. After performing the twist homeomorphism, this curve transforms into the one illustrated in Figure~\ref{dijagram35}b). This means that the attaching curve winds around the boundary torus again $p$ times along the meridian, but now $q+p$ times along the parallel. The framing of the upper circle at the right-hand side of Move~($2$) denotes this change.
\begin{figure}[h!]
    \centerline{\includegraphics[width=1\textwidth]{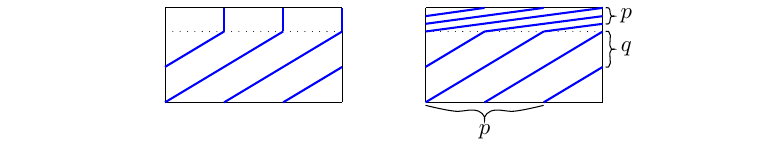}
    } \caption{a) and b)}\label{dijagram35}
\end{figure}

The situation inside the well is identical to the one discussed above in the case of Move~($-1$). This means that the lower circle, with framing $-1$, at the right-hand side of Move~($2$), according to Move~($-1$), covers this change in the surgery. The inverse of the twist homeomorphism produces the framing $p/(q-p)$ of the upper and $+1$ of the lower circle at the right-hand side of Move~($2$).

Move~($\infty$) just says that the trivial surgery could be performed anywhere without changing the manifold.

\vspace{2ex}
\noindent$(\Rightarrow)$ For this direction we show first that if two $\mathbb{Z}$-diagrams are $\mathbb{Z}$-equivalent, then they are $\mathbb{Q}$-equivalent. Since Moves~($-1$) and ($1$) exist in both calculi, it suffices to factor Move~($0$) through Moves~($\infty$) and ($2$) (see Figure~\ref{dijagram37}).
\begin{figure}[h!]
    \centerline{\includegraphics[width=.7\textwidth]{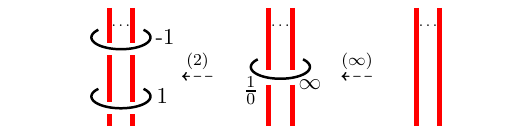}} \caption{}\label{dijagram37}
\end{figure}

Next we claim that for every $\mathbb{Q}$-diagram there exists a $\mathbb{Q}$-equivalent $\mathbb{Z}$-diagram. The proof of this fact proceeds in three steps as in the proof of \cite[Lemma~3]{R84}. For these steps, the wedges of circles are irrelevant and Rolfsen's proof covers our situation as well.

To finish the proof, let $\mathcal{D}_1$ and $\mathcal{D}_2$ be $\mathbb{Q}$-diagrams denoting $\partial$-equivalent manifolds. Let $\mathcal{D}'_1$ and $\mathcal{D}'_2$ be two $\mathbb{Z}$-diagrams, which are $\mathbb{Q}$-equivalent to $\mathcal{D}_1$ and $\mathcal{D}_2$ respectively. By $(\Leftarrow)$ direction of this proof, the manifolds denoted by $\mathcal{D}'_1$ and $\mathcal{D}'_2$ are $\partial$-equivalent. By Theorem~\ref{integral calculus}, the diagrams $\mathcal{D}'_1$ and $\mathcal{D}'_2$ are $\mathbb{Z}$-equivalent, and hence $\mathbb{Q}$-equivalent. Since $\mathbb{Q}$-equivalence is an equivalence relation, we have that $\mathcal{D}_1$ and $\mathcal{D}_2$ are $\mathbb{Q}$-equivalent.
\end{proof}


\begin{center}\textmd{\textbf{Acknowledgements} }
\end{center}
\smallskip
We would like to thank Danica Kosanovi\' c for very useful conversation concerning this topic. We are grateful to Nathalie Wahl for pointing out to us some weaknesses of the previous version of our text. We would also
like to extend our gratitude to the anonymous referee for a diligent reading of
our paper and thoughtful suggestions.

\medskip

\begin{center}\textmd{\textbf{Funding} }
\end{center}
\smallskip
Bojana Femi\' c, Jovana Obradovi\' c and Zoran Petri\' c were supported by the Science Fund of the Republic of Serbia,
Grant
No. 7749891, Graphical Languages - GWORDS. Vladimir Gruji\' c was supported by the Science Fund of the Republic of Serbia,
Grant
No. 7744592, Integrability and Extremal Problems in Mechanics, Geometry
and
Combinatorics - MEGIC.

\end{document}